\pdfoutput=1
\documentclass[12pt, reqno]{amsart}
\usepackage[T1]{fontenc}
\usepackage{graphicx}
\usepackage{amsmath}
\usepackage{amsfonts}
\usepackage{amssymb}
\usepackage{array}
\usepackage{xcolor}
\usepackage{colortbl}
\usepackage{cite}
\usepackage{xspace}
\usepackage{hyperref}

\numberwithin{figure}{section}
\numberwithin{table}{section}
\newtheorem{theorem}{Theorem}[section]
\newtheorem{lemma}[theorem]{Lemma}
\newtheorem{corollary}[theorem]{Corollary}
\newtheorem{conjecture}[theorem]{Conjecture}
\theoremstyle{definition}
\newtheorem{definition}[theorem]{Definition}
\newtheorem{remark}[theorem]{Remark}

\definecolor{Gray}{gray}{0.85}
\newcommand{\ddd}{\mbox{$\cdot\!\cdot\!\cdot$}}
\newcommand{\ie}{\textit{i.e.}\xspace}
\newcommand{\eg}{\textit{e.g.}\xspace}

\newcommand{\NN}{\mathbb{N}}
\newcommand{\mN}{\mathcal{N}}
\newcommand{\mC}{\mathcal{C}}
\newcommand{\ks}{k^{\star}}
\newcommand{\ys}{y^{\star}}
\newcommand{\tB}{H}

\begin{document}

\title[A dynamical approach towards Collatz conjecture]
      {A dynamical approach towards \\ Collatz conjecture}

\author{Pablo Casta\~neda}

\address{\vspace{-18pt}
\newline Pablo Casta\~neda
\newline Department of Mathematics, ITAM
\newline R\'io Hondo 1, Ciudad de M\'exico 01080, Mexico}
\email{pablo.castaneda@itam.mx}

\begin{abstract}
The present work focuses on the study of the renowned Collatz conjecture,
also known as the $3x+1$ problem. The analysis rests on the dynamics of an
iterative map in binary form, defined on the dyadic rationals of
$[1/2,\,1)$. We estimate how the length of the binary expansion of the
iterated numbers changes, and we regard the iterative map as a
perturbation of a circle map without periodic points, with an explicit
bound for the error between both maps. A condition on the periods of
periodic orbits is given; it is the classical argument of Crandall, and it
recovers the lower bound of Eliahou for the length of nontrivial cycles.
This version corrects several errors of the first one; they are listed in
the Addendum.
\end{abstract}


\maketitle

\section{Introduction}

\noindent
The Collatz conjecture is a long-standing open conjecture in number
theory. Among many other names, it is known as the conjecture for the
$3x+1$ problem, which concerns an arithmetic procedure over the integers.
The conjecture is based on the Collatz function given by
\begin{equation}\label{eq:collatz}
  C(x) \;=\;
  \begin{cases}
    3x+1, & \text{if }\; x \equiv 1 \pmod 2, \\
    x/2,  & \text{if }\; x \equiv 0 \pmod 2.
  \end{cases}
\end{equation}
The Collatz problem concerns the dynamical behavior of this map for every
positive integer $x$.

\begin{conjecture}[Collatz conjecture]
Starting from any positive integer $x$, the iterations of the function
$C(x)$ eventually reach the number $1$. Then they enter a cycle, taking the
successive values $1,\,4,\,2$.
\end{conjecture}

The problem has been addressed from several viewpoints for nearly a
century (see, \eg, \cite{Lag85,Lag10} for an overview), with approaches
from number theory, dynamical systems, ergodic theory, mathematical logic
and the theory of computation, as well as stochastic strategies. This
paper presents a hybrid numerical and dynamical analysis. The $3x+1$
problem is usually regarded as a discrete map $C:\mathbb{Z} \to
\mathbb{Z}$; the map proposed here sends the dyadic rationals of an
interval of $\mathbb{R}$ into themselves, although it is highly
discontinuous.

Sec.~\ref{sec:notation} introduces a binary function that captures the
dynamics of \eqref{eq:collatz} on the dyadic rationals of an interval, and
studies how the length of the binary expansion changes under one
iteration. In Sec.~\ref{sec:bin} this function becomes an iterative map,
which we compare with a circle map without periodic points; the
comparison gives an error bound along the orbits, and a separate argument
gives a condition on the periods of periodic orbits.
Sec.~\ref{sec:conclusions} discusses the results, and the Addendum lists
the changes with respect to the first version of this preprint.

\section{Some useful notation and results}
\label{sec:notation}

\noindent
The Collatz function can be simplified into more tractable functions. A
well-known case is the function $T(x) = x/2$ for even $x$ and $T(x) =
(3x+1)/2$ for odd $x$, studied by Everett and by Terras~\cite{Eve77,Ter76}.
The iteration is shortened further by the reduced Collatz function $R(x) =
(3x+1)/2^v$ for odd $x$, where $2^v$ is the largest power of $2$ dividing
$3x+1$; see~\cite{Col11}. Even though this reduced map simplifies the
Collatz procedure, it is not yet straightforward: $3x+1$ is even, but the
power of $2$ that divides it is unknown \emph{a priori}.

Let $x$ be an odd number with \textbf{binary expression} $(a_m a_{m-1}
\cdots a_0)_2$, where $a_k \in \{0,\,1\}$ for $k = 1,\,2,\,\dots,\,m-1$,
and $a_0 = a_m = 1$. In order to remove the computation of the unknown
power of $2$, let us consider the binary expression of $y = 2^{-(m+1)}x$,
\ie, $y = (0.a_m a_{m-1} \cdots a_0)_2$. Let us define the map
\begin{equation*}
M(x) = 2^{-\ell}x, \qquad \text{with $\;\ell\;$ the smallest natural
number such that } \;x < 2^\ell\; \text{ holds,}
\end{equation*}
so that $\ell = m+1$ and $M(x) = y$. Then $M(2^n x) = M(x)$ holds for all
$x,\,n \in \NN$, and $M$ is one to one from the odd natural numbers onto
the dyadic rationals of the real interval $I = [1/2,\,1)$. For instance,
$x = 1$ corresponds to $y = 2^{-1} = (0.1)_2$.

Now we need to calculate $3y + 2^{-\ell}$ for $y = (0.1 a_{m-1} a_{m-2}
\cdots a_1 1)_2$, which equals $(3x+1)/2^\ell$; the problem thus turns into
the control of the \textbf{total length} $\ell$ of $y$, that is, the
number of its binary digits after the binary point. On the other hand, the
benefit of iterating over $I$ is that $3y + 2^{-\ell}$ lies in the
interval $(3/2,\,3)$, as $y \le 1 - 2^{-\ell}$. Thus, in order to get a
value in $I$, the result must be divided by $2$ or by $4$. For instance,
$y = (0.1)_2$ has $\ell = 1$, and $3y + 1/2 = (1.1)_2 + (0.1)_2 = (10)_2 =
2$ holds; upon dividing by $4$ we get $(0.1)_2$, \ie, the fixed point
$1/2$, which corresponds to $x = 1$. In the same way, for $y = (0.10101
\cdots 01)_2$ we get $3y + 2^{-\ell} = (1.11 \cdots 11)_2 + (0.00 \cdots
01)_2$, which equals $2$. For a practical use of additions in vertical
form, see the sketch in Fig.~\ref{fig:add}; other uses of binary coding in
number theory can be found in~\cite{Cas14}, for instance.

\begin{figure}[htb]
\centering
\includegraphics[width = 0.8\textwidth]{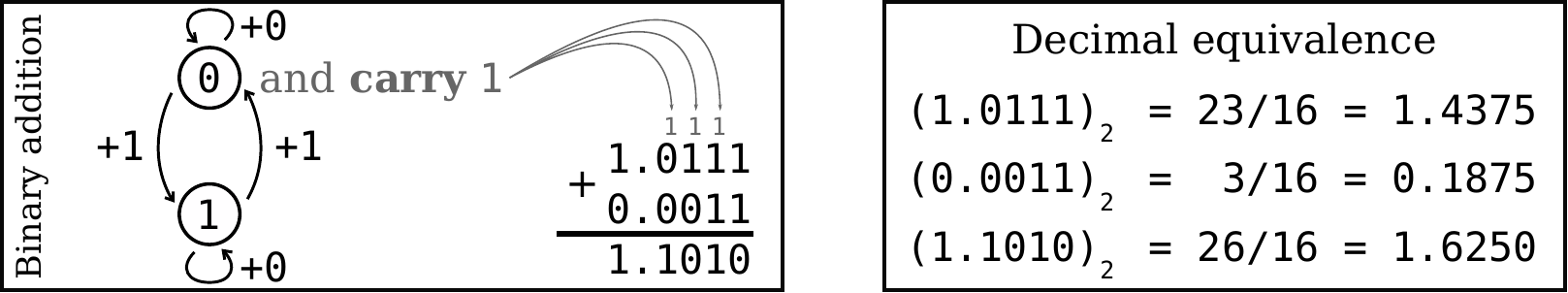}
\caption{\small Binary addition. On the left-hand panel, $(1)_2 + (1)_2 =
(10)_2$; thus, in vertical additions, every $1$ added to $1$ writes $0$
and carries $1$. The example adds $(1.0111)_2$ and $(0.0011)_2$. On the
right-hand panel, the same numbers in decimal notation.}
\label{fig:add}
\end{figure}

\begin{definition}
For $n \in \NN$, define $\ys_n = (0.1\{01\}^n)_2$, where $\{01\}^n$ stands
for $n$ repetitions of $01$. We call the set $\mN$ of all the numbers
$\ys_n$ the set of \textbf{predecessor numbers}, or simply
\textbf{predecessors}.
\end{definition}

\begin{remark}\label{rem:predecessors}
The number $x = 2^{2n + 1}\ys_n = (1\{01\}^n)_2 = 2^{2n} + 2^{2(n-1)} +
\cdots + 2^0$ satisfies $3x + 1 = 2^{2n+1} + 2^{2n} + \cdots + 2^0 + 1 =
2^{2(n+1)}$. Hence $C^{2n+3}(x) = 1$ and $R(x) = 1$ hold. (Notice that
$\ys_n$ tends to $2/3$ as $n \to \infty$.)
\end{remark}

In consequence, let us define the \textbf{binary function} by
\begin{equation}\label{eq:binary}
B(y) \;=\;
\begin{cases}
  1/2,                 & \text{if }\; y \in \mN, \\
  (3y + 2^{-\ell})/2, & \text{if }\; y \in [1/2,\,2/3) \setminus \mN, \\
  (3y + 2^{-\ell})/4, & \text{if }\; y \in (2/3,\,1),
\end{cases}
\end{equation}
for every dyadic rational $y \in I$, where $\ell$ is the total length of
$y$. (The value $y = 2/3$ never occurs, as $2/3$ is not dyadic.)

\begin{remark}
A point $y \in [1/2,\,1)$ without preimage $x \in \NN$ under $M$ (\eg, $y
= 4/5$) will follow the map \eqref{map:circle} below, since in such a case
$\ell$ is infinite and the term $2^{-\ell}$ vanishes.
\end{remark}

\begin{lemma}\label{lem:itself}
The binary function \eqref{eq:binary} maps the dyadic rationals of the
interval $I = [1/2,\,1)$ into themselves.
\end{lemma}

\begin{proof}
The proof consists of three arguments, one for each case
of~\eqref{eq:binary}.

(1) For $y \in \mN$, the result holds by definition and
Remark~\ref{rem:predecessors}.

(2) Let $y \in [1/2,\,2/3) \setminus \mN$ have total length $\ell$, so
that $x = 2^\ell y$ is an odd integer. As $y < 2/3$, the integer $3x$ is
smaller than $2^{\ell+1}$, so $3x \le 2^{\ell+1} - 1$, that is, $3y +
2^{-\ell} \le 2$ holds. Equality holds only for $x = (2^{\ell+1}-1)/3$,
which is the case $y \in \mN$ excluded here. Hence $3y + 2^{-\ell}$ lies in
$(3/2,\,2)$, and $(3y + 2^{-\ell})/2$ lies in $(3/4,\,1) \subset I$. In
binary terms, $y = (0.1 \{01\}^k 00 a_{m - 2k - 3} \cdots a_0)_2$ for some
$k \in \NN$, and the pair $00$ prevents the vertical addition of
Fig.~\ref{fig:add} from reaching $2$.

(3) For $y \in (2/3,\,1)$, the number $3y + 2^{-\ell}$ lies in $(2,\,3)$,
and in consequence $(3y + 2^{-\ell})/4$ lies in $(1/2,\,3/4) \subset I$,
which completes the proof.
\end{proof}

The binary function and the reduced function are equivalent: $B(M(x)) =
M(R(x))$ holds for every odd $x$, as both sides lie in $I$ and differ from
$3x + 1$ only by a power of $2$. Therefore, Remark~\ref{rem:predecessors}
and the proof of Lemma~\ref{lem:itself}, where $B(y) \neq 1/2$ for $y
\notin \mN$, yield the following result.

\begin{corollary}
For an odd $x \in \NN$, $R(x) = 1$ holds if and only if $M(x) \in \mN$.
\end{corollary}

The binary function~\eqref{eq:binary} also appears, with a different
notation, in the preprint of Colussi~\cite{Col17}, within his Collatz
process over the dyadic rationals, and in the work of Hew~\cite{Hew21},
who studied how the length of the binary expansion changes under $3x+1$.
I arrived at it independently (see Sec.~\ref{sec:conclusions}).

\begin{conjecture}[Binary Collatz conjecture]
For every odd $x \in \NN$, the orbit of $y = M(x) \in I$ under $B$ reaches
$1/2$. In other words, the fixed point $1/2$ of the binary
function~\eqref{eq:binary} attracts every dyadic rational of $I =
[1/2,\,1)$.
\end{conjecture}

For a given $y_0 \in [1/2,\,1)$, let us consider the length of $y_k =
B^k(y_0)$. The behavior of the length of binary expansions was
successfully analyzed for the baker's map in~\cite{BP08}; for the Collatz
function, this approach remains nonetheless unclear. The numbers in
Table~\ref{table:initial} compare the initial length in binary form with
the maximal length found in several runs, each one considering $500$
random numbers of a given length. The $200$ runs condensed in
Table~\ref{table:initial} give the maximal length and the maximal number
of iterations needed to achieve the \textbf{ground state} $1/2$, \ie, the
fixed point of $B$. We observe numerically that, in binary form, the
length behaves regularly.

\begin{table}[ht]
  \centering
    \caption{Maximal length and maximal number of iterations for random
    binary numbers in $I$ with a fixed initial length.}\label{table:initial}
    \begin{tabular}{r|rrrrrrr}
Initial length & $50$ & $100$ & $500$ & $1,000$ & $5,000$ & $10,000$ & $50,000$
\\ \hline
Maximal length & $+14$ & $+13$ & $+13$ & $+12$ & $+13$ & $+12$ & $+12$
\\ \hline
Maximal stopping time & $346$ & $449$ & $1,720$ & $3,136$ & $13,541$ &
$25,927$ & $124,514$
    \end{tabular}
\end{table}

\begin{remark}\label{rem:initial}
In Table~\ref{table:initial}, the column for numbers with $50,000$ binary
digits was obtained with ten runs of the routine iterating $500$ random
numbers of this length. Nine out of ten runs returned the record numbers
for the maximal length and the maximal stopping time, $50,012$ and
$124,514$, respectively. Only one run gave a maximal length of $+9$ binary
digits and $118,914$ as the maximal number of iterations before achieving
$1/2$.
\end{remark}

The stopping times of the reduced function $R(x)$ and of the binary
function $B(M(x))$ are equal. Nevertheless, Table~\ref{table:initial} and
our simulations indicate that, for random initial values, the binary form
has a regular behavior for the length of the expansion: an increase of
$14$ binary digits represents a factor of about $1.6 \times 10^{4}$ in
decimal notation, so a number in the first column of
Table~\ref{table:initial} may grow from $2^{50}\approx 1.1 \times 10^{15}$
to $2^{64} \approx 1.8 \times 10^{19}$. The \emph{hailstone numbers}
introduced in~\cite{Hay84} are iterates where there is an abrupt increase
of magnitude, say, from $x_n$ to $C(x_n)$. In Fig.~\ref{fig:trajectory},
we may compare the hailstone behavior of the decimal values with the
smoother behavior of their binary representations.

\begin{figure}[htb]
\centering
\includegraphics[width = \textwidth]{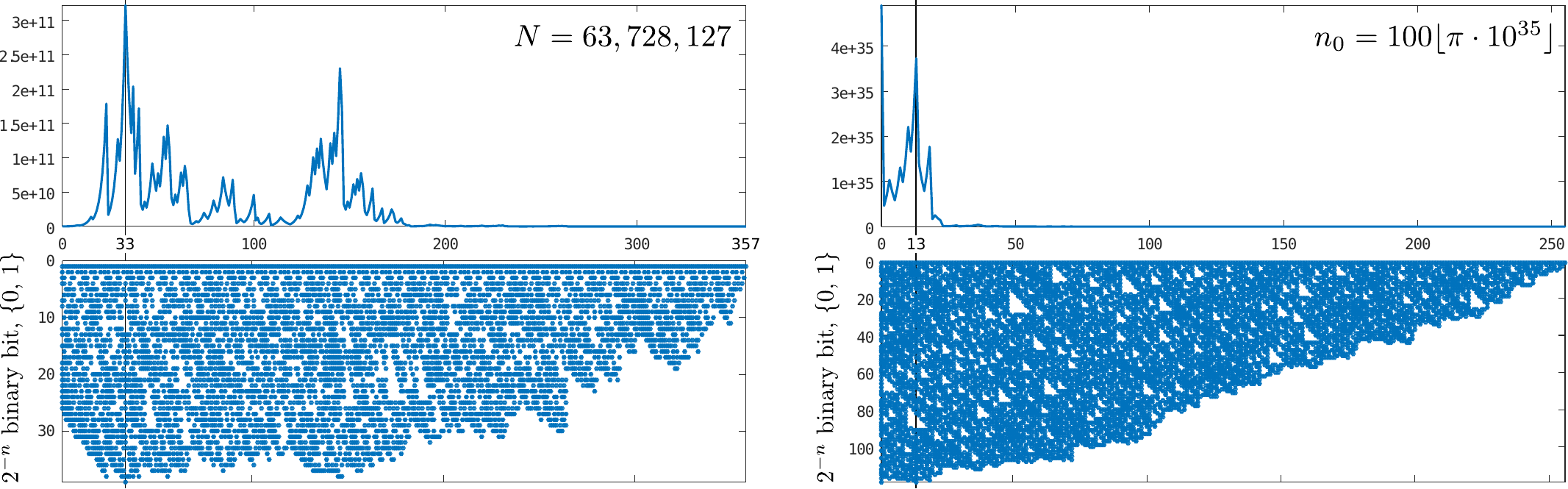}
\caption{\small Iterations of $x_0 = N := 63,728,127$ (left-hand panels)
and of $x_0$ the odd part of $n_0 := 100 \lfloor \pi \cdot 10^{35}
\rfloor$ (right-hand panels). Top panels: the iterates $x_n \in \NN$.
Bottom panels: the respective $y_n = M(x_n)$ in binary, with blank spaces
for zeros, dots for ones, and the top line standing for the binary point.
The number $n_0$ has $38$ decimal digits, and its odd part has $119$
binary digits; its orbit takes $255$ steps of $R$, against $529$ steps of
$T$ and $784$ steps of the Collatz function~\cite{Lag10,Roos}. Its
hailstone number $x_{13}$ has $36$ decimal digits and $119$ binary digits.
The number $N$ has $8$ decimal digits and $26$ binary digits; its
hailstone number $x_{33} = 322,205,345,153$ has $12$ decimal digits and
$39$ binary digits.}
\label{fig:trajectory}
\end{figure}

The slow growth of the string length is decisive for the Collatz
conjecture. Table~\ref{table:initial} suggests that, for random initial
values of length $m$, the longest expansion along the orbit has length
about $m + 14$. Let us estimate how the length changes from $y_0 = (0.1
a_{m-1} \cdots a_1 1)_2$ to $y_1 = B(y_0) = (0.1 b_{\tilde{m}-1} \cdots b_1
1)_2$. To do so, we follow the behavior of the ``head'' and of the
``tail'': we call \textbf{head} the block of the first $n$ binary digits,
and \textbf{tail} the block of the last $n$ binary digits. For example,
$(0.1001)_2$ is mapped to $(1.1100)_2$ by $y \mapsto 3y + 2^{-\ell}$, and
we count the change of length with the binary point as reference. The
length increases by one at the beginning, with a $1$ at the left of the
binary point, and decreases by two at the very end, because of the two
final zeros, since $y_0$ originally has four binary digits at the right of
the binary point. Thus the head adds one to the total length and the tail
subtracts two, and as a result of the iteration $y_1 = B(y_0)$ the total
length decreases by one. Let us focus on $n = 3$ binary digits, for which
there are four possible heads and four possible tails.

\begin{theorem}[Heads and tails]\label{thm:heads}
Under the binary function~\eqref{eq:binary}, the head with three binary
digits $h_1 = (0.100)_2$ adds one to the total length, and the heads $h_2
= (0.101)_2$, $h_3 = (0.110)_2$ and $h_4 = (0.111)_2$ add two. (A head
$h_2$ may add only one digit.) The tail with three binary digits $t_1 =
(\cdots 001)_2$ subtracts two from the total length, the tails $t_2 =
(\cdots 011)_2$ and $t_4 = (\cdots 111)_2$ subtract one, and the tail $t_3
= (\cdots 101)_2$ subtracts at least three.
\end{theorem}

\begin{proof}
There are eight cases to consider. As before, $*$ stands for an unknown
binary digit, and we use vertical additions as in Fig.~\ref{fig:add}, with
$3 = (11)_2 = (10)_2 + (1)_2$; for the heads, the addition of $2^{-\ell}$
remains implicit within the digits $*$. We obtain
\begin{equation*}
h_1: \begin{array}{r@{} @{}c@{} @{}r@{} @{}r@{} @{}r@{} @{}r@{}}
     0&.&1&0&0&*\cdots \\
+ \; 1&.&0&0&*&*\cdots \\
\hline
{\bf 1}&.&1&*&*&*\cdots
\end{array}\;,
\;\;
h_2: \begin{array}{r@{} @{}c@{} @{}r@{} @{}r@{} @{}r@{} @{}r@{} @{}r@{}}
     0&.&1&0&1&a&*\cdots \\
+ \; 1&.&0&1&a&b&*\cdots \\
\hline
{\bf 10}&.&*&*&*&*&*\cdots
\end{array}\;,
\;\;
h_3: \begin{array}{r@{} @{}c@{} @{}r@{} @{}r@{} @{}r@{} @{}r@{}}
     0&.&1&1&0&*\cdots \\
+ \; 1&.&1&0&*&*\cdots \\
\hline
{\bf 10}&.&*&*&*&*\cdots
\end{array}\;,
\;\;
h_4: \begin{array}{r@{} @{}c@{} @{}r@{} @{}r@{} @{}r@{} @{}r@{}}
     0&.&1&1&1&*\cdots \\
+ \; 1&.&1&1&*&*\cdots \\
\hline
{\bf 10}&.&1&*&*&*\cdots
\end{array}\;.
\end{equation*}
Thus $h_1$ adds one to the total length, and the other heads add two. If
$a$ and $b$ are zero for $h_2$, the sum is $(1.111*)_2$, and this head adds
only one to the total length.

Similarly, for the tails we add $x$ in the first row and $(10)_2x + 1$ in
the second row, and we obtain
\begin{equation*}
t_1: \begin{array}{@{}r @{\,}r@{} @{}r@{} @{}r@{} @{}r@{}}
     \cdots & *&0&0&1 \\
+ \; \cdots & 0&0&1&1 \\
\hline
\cdots&*&1&{\bf 0}&{\bf 0}
\end{array}\;,\quad
\quad
t_2: \begin{array}{@{}r @{\,}r@{} @{}r@{} @{}r@{} @{}r@{}}
     \cdots & *&0&1&1 \\
+ \; \cdots & 0&1&1&1 \\
\hline
\cdots & *&0&1&{\bf 0}
\end{array}\;,\quad
\quad
t_3: \begin{array}{@{}r @{\,}r@{} @{}r@{} @{}r@{} @{}r@{}}
     \cdots & *&1&0&1 \\
+ \; \cdots & 1&0&1&1 \\
\hline
\cdots & *&{\bf 0}&{\bf 0}&{\bf 0}
\end{array}\;,\quad
\quad
t_4: \begin{array}{@{}r @{\,}r@{} @{}r@{} @{}r@{} @{}r@{}}
     \cdots & *&1&1&1 \\
+ \; \cdots & 1&1&1&1 \\
\hline
\cdots & *&1&1&{\bf 0}
\end{array}\;.
\end{equation*}
Since the binary point stays $\ell$ positions to the left of these
additions, $t_2$ and $t_4$ subtract one from the total length, $t_1$
subtracts two, and $t_3$ subtracts at least three, which completes the
proof.
\end{proof}

\begin{remark}\label{rem:total}
From Theorem~\ref{thm:heads}, the contributions of all the three-digit
heads and tails to the total length are depicted in the following table.
Notice that the sum of all its cells is nonpositive.
\newcolumntype{R}[1]{>{\raggedleft\let\newline\\\arraybackslash\hspace{0pt}}m{#1}}
\begin{center}
\begin{tabular}{ l | R{2.05cm} R{2.05cm} | r | R{2.05cm} }
Heads $\backslash$ Tails &
$t_1 = (\ddd 001)_2$ & $t_2 = (\ddd 011)_2$ & $t_3 = (\ddd 101)_2$ & $t_4 = (\ddd 111)_2$ \\
\hline
$h_1 = (0.100\ddd\!)_2$ &     $-1\qquad$ &      $0\qquad$ & \cellcolor{Gray} $\leq -2\qquad$ &      $0\qquad$ \\
\hline
$h_2 = (0.101\ddd\!)_2$ & \cellcolor{Gray} $\leq 0\qquad$ & \cellcolor{Gray} $\leq 1\qquad$ & \cellcolor{Gray} $\leq -1\qquad$ & \cellcolor{Gray} $\leq 1\qquad$ \\
\hline
$h_3 = (0.110\ddd\!)_2$ &      $0\qquad$ &      $1\qquad$ & \cellcolor{Gray} $\leq -1\qquad$ &      $1\qquad$ \\
$h_4 = (0.111\ddd\!)_2$ &      $0\qquad$ &      $1\qquad$ & \cellcolor{Gray} $\leq -1\qquad$ &      $1\qquad$
\end{tabular}
\end{center}
\end{remark}

Remarks~\ref{rem:initial} and~\ref{rem:total} suggest not only that the
Collatz conjecture might be true, but also a way of analyzing the binary
structure at each iteration step. In the following section, we consider a
\emph{circle map} associated with the binary function~\eqref{eq:binary},
which improves the visualization of the iterates of $B(y)$.

\section{Associated binary map and orbits in the iterative map}
\label{sec:bin}

\noindent
We now analyze the binary function~\eqref{eq:binary} as an iterative map
of the dyadic rationals of $I = [1/2,\,1)$ into themselves.
One-dimensional dynamics is well understood when the map is a
diffeomorphism or a homeomorphism (see, \eg, \cite{dMvS93,MiTh88}). The
binary map $y_{n+1} = B(y_n)$, however, is highly discontinuous. For
example, let $y = (0.1011)_2$ and let $z = (0.101011 \cdots 1)_2$, which is
close to $y$. Upon applying the binary function, we have $B(y) =
(0.10001)_2$ and $B(z) = (0.10000011 \cdots 1)_2$. Thus the distance $|y -
z| = 2^{-\ell}$ depends on the length $\ell$ of $z$ and can be made
smaller than any $\delta > 0$, whereas $|B(y) - B(z)| > 2^{-6}$ holds
uniformly.

\subsection{A reduced circle map}
\noindent
By inspection, the discontinuous behavior of~\eqref{eq:binary} arises both
from the elements of $\mN$ and from the terms $2^{-\ell}$. Hence, we
consider $B(y)$ as a perturbation of the map
\begin{equation}
\label{map:circle}
\tB(y) \;=\;
\begin{cases}
  3y/2, & \text{if }\; y \in L \;:=\; [1/2,\,2/3), \\
  3y/4, & \text{if }\; y \in R \;:=\; [2/3,\,1).
\end{cases}
\end{equation}
When the endpoints of $I$ are identified, this map defines a
\emph{circle homeomorphism}, see~\cite{dMvS93}. It has a unique
discontinuity point $c = 2/3$ in $I$, where $\tB(c) = 1/2$ holds. It may
be extended to $\mathbb{R}$ by a continuous \emph{lifting}, with, \eg,
$\tilde{H}(1/2) = 3/4$. Also, $\tB$ is continuous and strictly increasing
on each component of $I\setminus\{c\}$, it is right-continuous at $c$, and
it maps both endpoints of $I$ to a single interior point, $3/4$. Moreover,
this map has no fixed points.

The orbit of a point $y$ is the set of its iterates, $\mathcal{O}(y) =
\{H^k(y)\}_{k \in \NN}$, where $\tB^k(y)$ stands for the $k$-th iterate of
$y$. We notice that
\begin{equation}\label{eq:kmap}
\tB^k(y) \;=\;
\begin{cases}
  3^ky/2^{\mu(k)},   & \text{if }\; y \in [1/2,\,2^{\mu(k)}\!/3^k), \\
  3^ky/2^{\mu(k)+1}, & \text{if }\; y \in [2^{\mu(k)}\!/3^k,\,1),
\end{cases} \qquad \text{ where }\; \mu(k) := \bigg\lfloor{k \frac{\ln
3}{\ln 2}}\bigg\rfloor.
\end{equation}
Thus the discontinuity point of $\tB^k$ is $c_k = 2^{\mu(k)}\!/3^k$. Since
$3^k = 2^p$ never holds for $k \ge 1$, $\tB^k$ has no fixed points, and in
consequence $\tB$ has no periodic points. In addition, $\tB$ is
\emph{topologically transitive}, as the orbit of any open interval
eventually covers the whole interval $I$. For a set $J$, we denote by
$\tB^{-1}(J)$ the set of all values $z \in I$ such that $\tB(z) \in J$,
and in the same way $B^{-1}(J)$ for the binary map.

Therefore, $\tB(y)$ is a well-defined circle map without chaotic behavior,
in contradistinction to the map $B(y)$. Nonetheless, we get information
from $\tB(y)$ in order to investigate the Collatz conjecture.

\subsection{The binary map as a perturbation}
\noindent
As an illustration of the dynamics, let us choose the set $T_5$ of all
numbers in $I$ that can be written with binary length at most $\ell = 5$.
This set is written with the first element $\tau_1 = 1/2$ and recursively
with $\tau_{i+1} = \tau_i + 2^{-5}$. Thus we have $T_5 = \{\tau_1 =
(0.10000)_2,\, \tau_2 = (0.10001)_2,\, \dots,\, \tau_{16} =
(0.11111)_2\}$, related to a disordered list of the odd numbers from $1$
to $31$ as $\tau_1 = M(1)$, $\tau_2 = M(17)$, $\tau_3 = M(9)$, $\tau_4 =
M(19)$, \dots, $\tau_{16} = M(31)$. Notice that $\tau_1 = (0.1)_2$, $\tau_5
= (0.101)_2$ and $\tau_6 = (0.10101)_2$ are in $\mN$, so their orbits under
$B$ are $\{\tau_i,\,\tau_1,\,\dots\}$. Direct computations show that
$B(\tau_i) \notin T_5$ only for $i = 8,\,12,\,16$; however, their orbits
under $B$ also reach $\tau_1$, as
\begin{align*}
 &\{\tau_8,\,(0.100011)_2,\,(0.110101)_2,\,\tau_5,\,\tau_1,\,\dots\}, \\
 &\{\tau_{12},\,(0.101001)_2,\,\tau_{16},\,\dots,\,\tau_1,\,\dots\},
\end{align*}
where the orbit of $\tau_{16}$ takes $39$ iterations to reach $\tau_1$. On
three occasions, the iterate $B^k(\tau_{16})$ has length $12$.

If we compute $\tB(y)$ instead of $B(y)$, the error can be estimated for
$y \in J := I \setminus (T_5 \cup \mN)$, since we know that the orbits
under $B$ of the points of $T_5 \cup \mN$ reach $1/2$. The length of any
$y \in J$ is $\ell \geq 6$, thus
$$\varepsilon_1 \;:=\; B(y) - \tB(y) \;\leq\; \left\{ \begin{array}{ll}
  2^{-7}, & \text{if }\; y \in L\cap J, \\
  2^{-8}, & \text{if }\; y \in R\cap J, \\
\end{array}\right.$$
and in any case $\varepsilon_1 \leq 2^{-7}$ holds. As $B(y)$ is larger
than $\tB(y)$, the absolute value is not needed.

We define the error $\varepsilon_k := B^k(y) - \tB^k(y)$ for $k \in \NN$
and $y \in I$. The orbit under $B$ is understood for those $y \in I$ such
that $B(y) \in T_5 \cup \mN$, so at each iteration we remove these points
from $I$. An inspection of the plots of $B$ and $\tB$ in
Fig.~\ref{fig:error} shows that $y \in L = [1/2,\,2/3)$ satisfies $3y/2 =
\tB(y) \approx B(y)$, and that $y \in R = [2/3,\,1)$ satisfies $3y/4 =
\tB(y) \approx B(y)$. Moreover, for $y \in L \setminus \mN$ both images
$B(y)$ and $\tB(y)$ lie in $R$, but the converse does not hold: for $y \in
R$ the image $\tB(y)$ lies in $[1/2,\,3/4)$ and the image $B(y)$ lies in
$(1/2,\,3/4)$, by Lemma~\ref{lem:itself}, and both intervals contain
points of $L$ and of $R$.

The alternation of the factors $3/2$ and $3/4$ is the same for $\tB$ and
$B$ once we remove the conflicting points $\gamma_k$ defined below.
Indeed, $\tB(y) \in L$ when $y \in [2/3,\,8/9)$, but for $y < 8/9$ close to
$8/9$ the image $B(y)$ may belong to $R$. As $8/9 =
(0.\{111000\}^\infty)_2$, where the block in braces repeats indefinitely,
and $y = (0.\{111000\}^k0*\cdots*1)_2 < 8/9$ implies $B(y) \in L$, we are
concerned with three types of numbers slightly smaller than $8/9$:
$$\alpha_k \;=\; (0.\{111000\}^k1)_2,\quad \beta_k \;=\;
(0.\{111000\}^k11)_2,\quad \gamma_k \;=\; (0.\{111000\}^k111)_2,$$
where $k$ is the number of repetitions of the block in braces. Direct
calculations give $B(\alpha_k) = (0.1\{01\}^{3k})_2 = \ys_{3k}$ and
$B(\beta_k) = (0.1\{01\}^{3k+1})_2 = \ys_{3k+1}$; thus $B^2(\alpha_k)$ and
$B^2(\beta_k)$ equal $(0.1)_2$, and $\alpha_k$, $\beta_k$ belong to the
preimage of $\mN$ under $B$; see Remark~\ref{rem:predecessors}. As
$B(\gamma_k) > 2/3$ holds, we denote by $\mC$ the set of all the
$\gamma_k$.

Let $T_\ell$ be the set of all $y = M(x)$ with $x$ odd and $x < 2^\ell$,
that is, of all the binary expressions with length at most $\ell$. We say
that $T_\ell$ is a \textbf{tested set} if the orbit under $B$ of every $y
\in T_\ell$ contains $\tau_1 = 1/2$. Then define $J_1 := I \setminus
(T_\ell \cup \mN \cup B^{-1}(\mN) \cup \mC)$ and, recursively, $J_n := J_1
\cap B^{-1}(J_{n-1})$, so that a dyadic rational $y$ belongs to $J_n$ when
$y,\,B(y),\,\dots,\, B^{n-1}(y)$ belong to $J_1$.

\begin{theorem}\label{thm:error}
Let $y \in J_n$, and assume that $B^j(y)$ and $\tB^j(y)$ lie in the same
subinterval, $L$ or $R$, for $j = 0,\,1,\,\dots,\,n-1$. Then the error
$\varepsilon_n := B^n(y) - \tB^n(y)$ satisfies
$$0 \leq \varepsilon_{2n} \leq 7 \left[\left(\frac{9}{8}\right)^{\!\!n} -
1\right] 2^{-\ell}, \qquad 0 \leq \varepsilon_{2n+1} \leq \left(
\frac{15}{2} \left[\left(\frac{9}{8}\right)^{\!\!n} - 1\right] +
\frac{1}{2}\right) 2^{-\ell}.$$
\end{theorem}

\begin{proof}
In order to show this result, we highlight two facts. Firstly, for $z \in
J_1$ the difference between $B(z)$ and $\tB(z)$ is $2^{-\ell_z}/2^i$,
where $\ell_z \ge \ell + 1$ is the length of $z$ and $i$ is equal to $1$
or $2$; for $y \in J_n$, this holds for each of the points $z = B^j(y)$
with $j < n$. Secondly, by hypothesis $B^j(y)$ and $\tB^j(y)$ lie in the
same subinterval, $L$ or $R$, so both are divided by the same power
$2^{i_{j+1}}$ at each step. Therefore, in the expansion of the difference
between $B^n(y)$ and $\tB^n(y)$ the terms in $y$ cancel; for instance,
\begin{align*}
\varepsilon_3 &= \frac{\displaystyle 3 \frac{\displaystyle 3 \frac{3y +
2^{-\ell_1}}{2^{i_{1}}} + 2^{-\ell_2}}{2^{i_2}} + 2^{-\ell_3}}{2^{i_3}} -
\frac{3^3 y}{2^{\mu(3) + a}} \\
 &= \frac{\displaystyle 3 \frac{\displaystyle 3 \frac{2^{-\ell_1}}{2^{i_1}} +
2^{-\ell_2}}{2^{i_2}} + 2^{-\ell_3}}{2^{i_3}},
\end{align*}
where $a$ is either $0$ or $1$ (see Eq.~\eqref{eq:kmap}), and $i_1 + i_2 +
i_3 = \mu(3) + a$ holds.

Now, $2^{-\ell_j} \leq 2^{-\ell}$ holds at each iteration. The powers $i_1,\,
i_2,\,\dots$ go from $2$ to $1$ or $2$, and from $1$ to $2$, as $y \in L$
gives $B(y) \in R$; since each term is larger when the powers are smaller,
the expansion is largest when the powers alternate as $1,\,2,\,\dots,\,
1,\,2,\,1$, where the last one must be $1$. Setting all the lengths equal
to $\ell$ and splitting into odd and even iterations, we get
\begin{align*}
\varepsilon_{2n+1} &\leq 2^{-1}(3\cdot 2^{-2}(\cdots(3\cdot2^{-2}(3\cdot2^{-1} + 1) + 1)\cdots) + 1) \, 2^{-\ell} \\
 &= \left( \left[ \sum_{k=1}^n 3^{2k}2^{-(3k+1)} + 3^{2k-1}2^{-3k} \right] + \frac{1}{2} \right) 2^{-\ell}
 \;=\; \left( \frac{15}{2} \left[\left(\frac{9}{8}\right)^{\!\!n} - 1\right] + \frac{1}{2}\right) 2^{-\ell}, \\
\varepsilon_{2n}   &\leq 2^{-1}(3\cdot 2^{-2}(\cdots(3\cdot2^{-1}(3\cdot2^{-2} + 1) + 1)\cdots) + 1) \, 2^{-\ell} \\
 &= \left[ \sum_{k=1}^n 3^{2k-1}2^{-3k} + 3^{2(k-1)}2^{-(3k-2)} \right]2^{-\ell}
 \;=\; 7 \left[\left(\frac{9}{8}\right)^{\!\!n} - 1\right] 2^{-\ell},
\end{align*}
which completes the proof.
\end{proof}

\begin{figure}[htb]
\centering
\includegraphics[width = 0.32\textwidth]{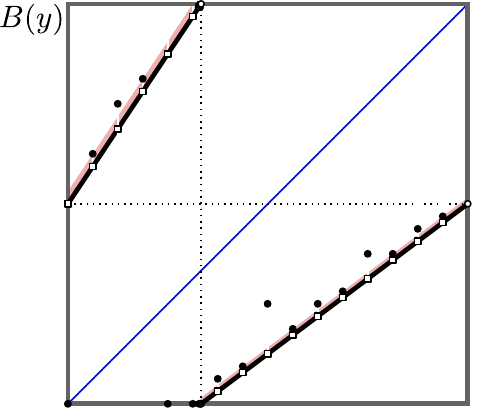}
\includegraphics[width = 0.32\textwidth]{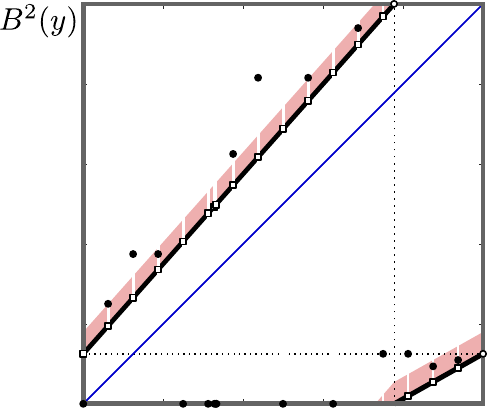}
\includegraphics[width = 0.32\textwidth]{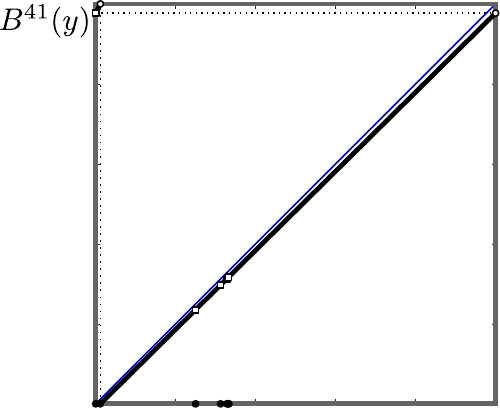}
\caption{\small Samples of $B^k(y)$ as a perturbation of $\tB^k(y)$, for
$k = 1,\,2,\,41$ (panels (a), (b) and (c)); shaded regions indicate where
the values of the binary map can lie, drawn twice as wide for graphical
purposes. Panels (a) and (b) show the exact values of $B^k(\tau_i)$ for
$\tau_i \in T_5$. For $k = 41$, $c_k$ is close to $1/2$, so $\tB^{41}$ is
close to the identity; the next iterate is far from it, as $c_{42}$ equals
$0.674352\dots$.}
\label{fig:error}
\end{figure}

The map $\tB$ has no periodic points. However, as $c_k$ can be close to
$1/2$, as shown in Fig.~\ref{fig:error}(c), the graph of $\tB^k$ may be
close to the diagonal, and the comparison with $\tB$ alone does not
exclude periodic points of $B$ of period $k$. The following result gives
a condition on such periods.

\begin{lemma}\label{lem:tested}
Assume that $T_\ell$ is a tested set. Then $B$ has no periodic orbit other
than $\{1/2\}$ with period $k < \ks$, where $\ks$ is the first natural
number for which some $p \in \NN$ satisfies
\begin{equation}\label{eq:period}
0 \;<\; p - k\log_2 3 \;<\; \frac{k}{3\cdot 2^{\ell}\ln 2}.
\end{equation}
\end{lemma}

\begin{proof}
Let $y_1 = M(x_1)$ be a periodic point of $B$ of period $k$ whose orbit
$y_1,\,\dots,\,y_k$, with $y_j = M(x_j)$, is not $\{1/2\}$. The orbit does
not meet $T_\ell$, since the orbit of a point of a tested set reaches
$1/2$; hence $x_j > 2^\ell$ holds for every $j$. It does not meet $\mN$
either. For $z \in I \setminus \mN$ of length $\ell_z$ and $x =
2^{\ell_z}z$, we have $3z + 2^{-\ell_z} = 3z\,(1 + 1/(3x))$, so that $y_{j+1}
= 3y_j(1 + 1/(3x_j))/2^{e_j}$ with $e_j \in \{1,\,2\}$. After $k$ steps,
$$y_1 \;=\; y_1\,\frac{3^k}{2^p} \prod_{j=1}^{k}\Big(1 + \frac{1}{3x_j}\Big),
\qquad \text{where } p := e_1 + \cdots + e_k.$$
Taking logarithms in base $2$, we obtain $k\log_2 3 + \sum_{j}\log_2(1 +
1/(3x_j)) = p$. Each term of the sum is positive and smaller than $1/(3x_j
\ln 2) < 1/(3\cdot 2^{\ell}\ln 2)$, so \eqref{eq:period} holds, and therefore
$k \ge \ks$, which completes the proof.
\end{proof}

\begin{remark}\label{rem:periods}
Condition~\eqref{eq:period} asks for a very good rational approximation
$p/k$ of $\log_2 3$ from above, and such approximations come from the
continued fraction of $\log_2 3$; this is essentially the argument of
Crandall~\cite{Cra78}, see also~\cite[Sec.~2.6]{Lag85}. As the convergence
test has been double checked for all $x \leq 2^{60}$~\cite{Roos}, we may
take $\ell = 60$ in Lemma~\ref{lem:tested}, which gives $\ks =
397,573,379$: there are no periodic orbits of $B$ other than $\{1/2\}$
with fewer elements. With the verification available to Eliahou, $\ell =
40$, one gets $\ks = 10,781,274$ with $p = 17,087,915$. Over a whole
cycle, $p$ is the total number of divisions by $2$ of $R$, so the
corresponding cycle of $T$ has $p$ elements, and this recovers the lower
bound $17,087,915$ of Eliahou~\cite{Eli93}. Sharper methods give larger
bounds~\cite{Her23,SdW05}.
\end{remark}

\section{Conclusions}
\label{sec:conclusions}

\noindent
As Lagarias puts it, ``although the problem began life as a curiosity,
its general connection with various other areas of mathematics'' has
``made it a respectable topic for mathematical research''~\cite{Lag10}.
Here we have presented a link with one-dimensional dynamics: the binary
function $B$ is a perturbation of a circle map $\tB$ without periodic
points. In~\cite{Lag10}, ``part of the difficulty appears to reside in an
inability to analyze the pseudorandom nature of successive iterates''; in
the present approach, the deviation of the iterates of $B$ from the
regular dynamics of $\tB$ is controlled by the error bounds of
Theorem~\ref{thm:error}, as long as both orbits visit the same
subintervals.

We have highlighted the behavior of the length of the binary expansions,
which seems to smooth out the hailstone numbers of~\cite{Hay84}, as well
as a procedure that, for a given set of tested numbers, eliminates periods
of periodic orbits. This procedure is the argument of
Crandall~\cite{Cra78}, and it recovers the lower bound of
Eliahou~\cite{Eli93} for the length of nontrivial cycles; it does not
exclude them.

During the revision of the first version of this manuscript, Hew's
article~\cite{Hew21} was published. It shows that the binary map $B(y)$
in~\eqref{eq:binary} is discontinuous at every point and that its graph
has a fractal structure, when $B$ is regarded as a map on the dyadic
rationals of $[1/2,\,1)$, \ie, on the odd natural numbers via the map
$M(x)$. Hew also describes exactly how the length changes, which completes
the analysis of Theorem~\ref{thm:heads}: Hew removes the final repetitions
of $01$, so that instead of $y = (0.1* \cdots *11\{01\}^k)_2$ one takes
$\hat{y} = (0.1* \cdots *11)_2$, and only the tail $t_3$ is excluded.
Then, in the domain $L$ only the heads $h_1$ and $h_2$ are possible, each
adding one digit, so the string length shrinks or keeps the same length;
in the domain $R$ the heads $h_2$, $h_3$ and $h_4$ add two digits, so the
string length grows or keeps the same length.

\section*{Addendum: changes in version 2}

\noindent
The first version of this preprint (arXiv:1910.08497v1) remains available.
The English of this version was revised with the assistance of Claude, an
artificial intelligence assistant by Anthropic, which also checked the
computations and proposed the corrections listed below. The ideas are those
of the first version, except where a correction required a new argument
(items~4 and~12).

\begin{enumerate}
\item The abstract and the introduction were rewritten to describe the
corrected results.
\item Arithmetic in Sec.~\ref{sec:notation}. For $y = (0.1)_2$, version 1
stated $(1.1)_2 + (0.1)_2 = (10)_2 = 4$ and divided by $8$; the sum is $2$
and it is divided by $4$, giving the fixed point $1/2$ (version 1 called
it ``the fixed point $1$''). The same correction applies to $y = (0.10101
\cdots 01)_2$, and a value of $I$ is obtained upon dividing by $2$ or $4$,
not by $2$, $4$ or $8$. The intervals $(3/2,\,4]$ and $(2,\,4)$ of version
1 are not wrong, and they were sharpened to $(3/2,\,3)$ and $(2,\,3)$.
Also, $2^{64} \approx 1.8 \times 10^{19}$ holds, not $2 \times 10^{20}$.
\item Notation and statements. Version 1 defined $M(x)$ with two different
meanings of $m$; the total length $\ell$ is now used throughout. An
incorrect remark stating $R(x) = R(2^n x)$ was removed. The binary
function is now defined on the dyadic rationals, with middle case
$[1/2,\,2/3)$ instead of $[1/2,\,2/3]$; the corollary is stated for odd
$x$; the binary Collatz conjecture now says that orbits reach $1/2$
(version 1 wrote $B^k(y) \to 1$). All statements now share a single
counter.
\item Lemma~\ref{lem:itself}. In case (2), version 1 showed that $3y +
2^{-\ell}$ cannot exceed $2$ but did not exclude the value $2$, which
would give $B(y) = 1 \notin I$; an arithmetic proof is given, and the
binary argument is kept as an interpretation. In case (3), the binary
argument was replaced by a direct bound.
\item Version 1 presented the binary function as new. It also appears in
the preprint of Colussi~\cite{Col17} and in the article of
Hew~\cite{Hew21}; this is now acknowledged in Sec.~\ref{sec:notation}.
\item The statement that the longest expansion has length about $m + 14$
now refers to random initial values, as in Table~\ref{table:initial}; for
special values it fails (for $x = 2^m - 1$ the excess is about
$(\log_2 3 - 1)m$). The phrase ``may provide a useful guide to its
proof'' was replaced by the numerical observation itself.
\item Caption of Fig.~\ref{fig:trajectory}. The right-hand panels show the
odd part of $n_0$, which has $119$ binary digits ($n_0$ itself has $125$).
The $529$ steps correspond to the function $T$; the Collatz function takes
$784$ steps. The hailstone number $x_{13}$ has $36$ decimal and $119$
binary digits, not $35$ and $116$.
\item Proof of Theorem~\ref{thm:heads}. The second row of the tail
additions is $(10)_2x + 1$, not $(10)_2x + 2^{-\ell}$.
\item Sec.~\ref{sec:bin}. The points with $B(\tau_i) \notin T_5$ are
$\tau_8$, $\tau_{12}$ and $\tau_{16}$, and the two orbits displayed are
those of $\tau_8$ and $\tau_{12}$; version 1 listed $\tau_7$, $\tau_{11}$
and $\tau_{16}$. Version 1 wrote $\tB^k(\tau_{16})$, $L\setminus J$,
$R\setminus J$ and $(100011)_2$ where $B^k(\tau_{16})$, $L \cap J$, $R
\cap J$ and $(0.100011)_2$ were meant, and it said that the orbits of
$T_5 \cup \mN$ under $\tB$ reach $1/2$, which holds for the orbits under
$B$. The statement on the images of $R$ was corrected, since $B(y)$ lies
in $(1/2,\,3/4)$ by Lemma~\ref{lem:itself}. The reference to
Fig.~\ref{fig:error}(b) for $c_k$ close to $1/2$ was removed, as $c_2 =
8/9$. The circle map is no longer said to define an inverse map
$\tB^{-1}: I \to I$.
\item Sets $T_\ell$ and $J_n$. Version 1 used three conventions for the
length of the tested set; now $T_\ell$ contains the binary expressions of
length at most $\ell$. Version 1 defined $J_n := J_{n-1} \setminus
B^{-1}(J_{n-1})$, which removes the points that should be kept; now $J_n
:= J_1 \cap B^{-1}(J_{n-1})$.
\item Theorem~\ref{thm:error}. The bounds are those of version 1. Version
1 derived that $B^j(y)$ and $\tB^j(y)$ lie in the same subinterval from
the exclusion of $\mC$ and of the preimages of $\mN$; this is false. For
instance, with $\ell = 5$ and $y = M(47)$, all the iterates stay in $J_1$,
but $B^{10}(y) \approx 0.6836$ lies in $R$ while $\tB^{10}(y) \approx
0.6617$ lies in $L$. The property is now an explicit hypothesis. In the proof, the exponent $2^{-3k+1}$ of the sum for $\varepsilon_{2n}$ was corrected to $2^{-(3k-2)}$; the closed forms were already right. (Along
an orbit, the error is in fact much smaller than these bounds, which
assume the worst alternation of the factors $3/2$ and $3/4$ at every
step.)
\item Lemma~\ref{lem:tested} and Remark~\ref{rem:periods}. Version 1
excluded periodic orbits of period $k < \ks$, with $\ks$ the first $k$
such that $1/2 + \varepsilon_k > c_k$, and obtained $\ks = 600$ for $\ell
= 60$. That computation is correct under its own criterion, but the lemma
was not proved: it used the same-subinterval property of item~11 along a
periodic orbit, and Conjecture~3 of version 1 to exclude $\mC$. The new
proof sums the logarithms along the cycle; it gives $\ks = 397,573,379$
for $\ell = 60$, and it does not use the error bound. This is the
classical argument of Crandall~\cite{Cra78}, and for $\ell = 40$ it
coincides with the bound of Eliahou~\cite{Eli93}; no novelty is claimed
for it.
\item Conjecture~3 of version 1 and the remark that followed it were
removed. Conjecture~3 stated that the orbit of every point of $\mC$
reaches $1/2$, which is a particular case of the Collatz conjecture, and
the new proof of Lemma~\ref{lem:tested} does not need it. The remark
concerned points $z = n/3^k$, which are not dyadic, so it was ill-posed.
With them, version 1's remark that there are no periodic points in $\mC$
and its conclusion that the problem reduces to showing that $\mC$
satisfies the Collatz conjecture were withdrawn.
\item Sec.~\ref{sec:conclusions}. Two sentences of version 1 were taken
nearly verbatim from Lagarias's overview without quotation marks; they are
now quoted and cited~\cite{Lag10}. Hew's results are described more
precisely: version 1 said that the binary map ``is actually a fractal''.
Version 1 also stated that in the domain $L$ the string length can grow;
under $B$ it cannot, as every head in $L$ adds exactly one digit and every
tail subtracts at least one.
\item References. Hew's article now has its final bibliographic data, the
two works of Colussi are listed separately, and the works of Crandall,
Eliahou, Lagarias (1985), Simons and de Weger, and Hercher were added.
\end{enumerate}

\section*{Acknowledgments}
\noindent
I am grateful to V\'ictor F. Bre\~na-Medina (ITAM) for detailed readings
and corrections of the first version, and to the anonymous DCDS reviewer
for enlightening comments. I was partially supported by Asociaci\'on
Mexicana de Cultura A.C.

\end{document}